\documentclass[10pt,leqeq]{article}
\usepackage{amssymb,amsthm,mathrsfs}
\usepackage[utf8]{inputenc}
\usepackage{youngtab}
\usepackage{amsmath}
\usepackage{amscd}
\usepackage{xy}
\xyoption{all}
\usepackage[dvips]{graphicx}
\usepackage{graphpap}
\usepackage{youngtab}
\usepackage{pstricks}
\usepackage{pst-all}
\usepackage{pstcol}
\usepackage{color}
\usepackage[T1]{fontenc}
\usepackage{tgbonum}
\usepackage{tikz}
\usetikzlibrary{decorations.pathreplacing}

\usepackage{ytableau}
\newtheorem{thm}{Theorem}
\newtheorem{lem}[thm]{Lemma}
\newtheorem{cor}[thm]{Corollary}
\newtheorem{prop}[thm]{Proposition} 
\newtheorem{rem}[thm]{Remark}

\newtheorem{exmp}[thm]{Example}

\date{}
\begin{document}
\setlength{\baselineskip}{16pt}
\title{A Continuous Analogue for Young Diagrams}
\author{Rafael D\'iaz}
\maketitle

\begin{abstract}
We build a continuous analogue for Young diagrams, thought of as left-aligned stairs, following the line of research initiated by D\'iaz and Cano on the construction of continuous analogues for combinatorial objects.
\end{abstract}

\section{Introduction}

In this work we continue our program to uncover continuous analogues for combinatorial objects with the methodology introduced in joint work with Cano \cite{cd, dc}, \ and futher developed and applied by Cano and Carrillo \cite{cc}, \ D\'iaz \cite{d}, \
Le, Robins, Vignat and Wakhare \cite{lrvw}, \ O'Dwyer \cite{rod}, \ Vignat and Wakhare \cite{vw}. This methodology is based on four main ideas:

\begin{itemize}
   \item It is often the case that a combinatorial object comes with a bijection with the lattice points of a convex polytope. Of course such a bijection need not be unique, nevertheless, in some cases it does arise in a seamless fashion. In any case, our construction of continuous analogues depends on the choice of such  bijections. 

  \item There is a deep, albeit subtle, relation between counting lattice points and computing the volume of a convex polytope. This remark have a long history, and have been developed by several authors through the years.  For more on this topic the reader may consult
      Barvinok \cite{b}, Beck and Robins \cite{ber}, Brion \cite{br}, De Loera \cite{de}, Ehrhart \cite{e}, Guillemin \cite{g}, and the references therein. 
  
  \item Sometimes a combinatorial object comes with a bijection with the lattice points of a finite family of convex polytopes of various dimensions; in the construction of the continuous analogue it may be more natural to consider countable families of such polytopes.
      
  \item  Thus, in the positive case where the above considerations apply, the continuous analogue of the cardinality of a finite set turns out to be the convergent infinite sum of the volume of polytopes of various dimensions depending on some continuous parameters.
  
  \end{itemize}

\

In this work we focus on finding a continuous analogue for Young diagrams, which are combinatorial objects with applications in the theory of numerical partitions \cite{ae}, and in the representation theory of permutation groups  \cite{ap}, among others. If one thinks of a Young diagrams as left-aligned stairs built with blocks of integral length and height $\ 1, \ $ the continuous analogue considered in this work are left-aligned stairs built with blocks  whose length and height are non-negative real numbers. It is clear that there are an uncountable number of such stairs, however, we are going to show that these continuous stairs come with enough structure so that their volume rather than their cardinality can be measured.

\

\section{A continuous analogue for compositions}

We begin constructing a continuous analogue for compositions within the settings outlined in the introduction. For the purposes of this work, this case is of interest primarily as an illustration of the techniques and conventions to be used in other sections.\\

For $\ x\in \mathbb{R}_{> 0}, \ k \in  \mathbb{N}_{> 0}\ $ let the $\ k $-simplex $\ \Delta^k_x \subseteq\mathbb{R}^k\ $ be the convex polytope such that $\ (x_1,\dots,x_k) \in \Delta^k_x\ $  
if and only if  $\  0\leq x_1 \leq \cdots \leq x_k \leq x. \ $ 
The simplex $\ \Delta^k_x\ $  inherits a Riemannian metric from $\ \mathbb{R}^k,\ $  which gives rise to a volume form, and the computation of integrals. The $\ k $-simplex $\ \Delta^k_x \ $ is often given in affine coordinates as the set of tuples $\ (z_0,\dots,z_k)\in \mathbb{R}_{\geq 0}^k\ $ such that  $\ z_0+ \cdots + z_k=x .\ $ In such cases we shall always work with Cartesian coordinates $\ (x_1,\dots,x_k)\ $ given by $\ x_i=z_0+ \cdots + z_{i-1}. \ $   
Affine coordinates can be recovered from Cartesian coordinates as $\ z_i= x_{i+1}-x_i \ $ where we set  $\ x_0=0, \  x_{k+1}=x. \ $  We let 
$\ \Delta^0_x \ $ be a single point, from the affine representation we see that such point may be identified with  $\ \{x\}.$\\

Recall \cite{ae} that a numerical composition of length $\ k+1\ $ of positive integer $\  n \ $  is a sequence $\ (l_0,\dots,l_k)\ $ of positive integers such that $\ l_0+ \cdots + l_k=n .\ $ Using Cartesian coordinates 
$\ n_i=l_0+ \cdots + l_{i-1} ,\ $ we see that compositions are in bijective correspondence with  tuples  of integers $\ 0 < n_1 < \cdots < n_k < n.\ $ In other words, compositions are in bijective correspondence with the interior ( not in the boundary ) integral points of  the convex polytope $\ \Delta^k_n \subseteq\mathbb{R}^k. \ $ \\

We think of the volume $\ \displaystyle \mathrm{vol}(\Delta^k_x)= \frac{x^k}{k!}\ $ of $\ \Delta^k_x \ $ as a continuous analogue for the number 
$\ \displaystyle c_{k+1}(n)= \binom{n-1}{k} \ $ of compositions of length $\ k+1 \ $ of $\ n. \ $  The total number of compositions is given by  
$$\ c(n)\ = \  \sum_{k=0}^{n-1} c_{k+1}(n) \ = \ \sum_{k=0}^{n-1}\binom{n-1}{k} \ = \ 2^{n-1}.  $$ 
A continuous analogue $\ \kappa(x) \ $ for $\ c(n) \ $ is built from the replacements
$\ n >0 \ \rightarrow \ x>0, \  $ $\ c_{k+1}(n) \ \rightarrow \ \mathrm{vol}(\Delta^k_x) ,\ $ $\ \sum_{k=0}^{n-1}\ \rightarrow \ \sum_{k=1}^{\infty}, \ $
thus for $\ x\in \mathbb{R}_{> 0}\ $ the function $\ \kappa(x) \ $  is given by
$$\kappa(x)\ = \ \sum_{k=1}^{\infty}\mathrm{vol}(\Delta^k_x) \ = \  \sum_{k=1}^{\infty}\frac{x^k}{k!} \ =  \ e^{x}-1,$$
i.e. the continuous analogue for the discrete exponential  
$\ 2^{n-1} \ $ turns out to be the continuous exponential
$\ e^{x}-1. \ $ Note that the exponential generating function of the composition numbers  $\ c(n) \ $ also gives rise to the continuous exponential function, indeed
$$\sum_{n=1}^{\infty}c(n)\frac{x^{n} }{n!} \ = \ \frac{1}{2} \sum_{n=1}^{\infty}\frac{(2x)^{n}}{n!} \ = \ \frac{1}{2} (e^{2x}-1)
\ = \ \frac{\kappa(2x)}{2}  .$$

\

\section{Integrals on the $\ k$-simplex $\ \Delta^k_x$} 

In this section we review some useful facts on integration on the $\ k$-simplex $\ \Delta^k_x . \ $
Just as $\ \mathrm{vol}(\Delta^k_x) \ $ is though of as a continuous analogue for the composition numbers  $\  c_{k+1}(n), \ $  we regard integrals of the form

$$\int_{\Delta^k_x}f(x_1,\dots,x_k)dx_1\dots dx_k \ \ = \ \ 
\int_{0}^{x}\bigg(\int_{\Delta^{k-1}_{x_{k}}}f(x_1,\dots,x_k)dx_1\dots dx_{k-1} \bigg) dx_k  $$

\noindent as continuous analogues for sums of the form $$\sum_{\mathbb{Z}^k \cap \Delta^{k,\mathrm{open}}_n} f(n_1,\dots,n_k)\ \ = \ \ \sum_{0< n_1 < \cdots < n_k<n} f(n_1,\dots,n_k).$$

\

\noindent \textbf{Notation:} For $\ a=(a_1,\dots,a_k), \ $  set $\ x^a=x_1^{a_1}\cdots x_k^{a_k},\ $
$ dx=dx_1\cdots dx_k, \ $ $ a!=a_1!\cdots a_k!,\ $ $ x^{(a)}=\frac{x^a}{a!} \ $ so  $\ x^{(a)}x^{(b)}=\binom{a+b}{b}x^{(a+b)}. \ $ For $\ 1 \leq i \leq k-1 \ $
set $\ |a|_i=a_1+\cdots + a_i, \ $ and $|a|=|a|_k = a_1+\cdots + a_k.\ $
Let $\ (x)^{(k)}=x(x+1)\cdots (x+k-1)\ $ be the raising factorial, and 
$\ (x)_{(k)}=x(x-1)\cdots (x-k+1)\ $  be the lowering factorial.\\

\begin{lem}\label{l1}
{\em For $\ k \in \mathbb{N}_{\geq 1}, \ a=(a_1,\dots,a_k)\in \mathbb{N}^k, \ $ we have that
$$\int_{\Delta^k_x}x_1^{a_1}\dots x_k^{a_k}dx_1\dots dx_k \ \ = \ \ 
\frac{x^{|a| + k}}{\prod_{i=1}^{k}\big(|a|_i+i\big)} \ \ \ \ \ \ \ \ \ \ \ \ \ $$
$$\int_{\Delta^k_x}x_1^{(a_1)}\dots x_k^{(a_k)}dx_1\dots dx_k \ \ = \ \ 
\prod_{i=2}^{k}\binom{|a|_i+i-1}{a_i}\ x^{(|a|+k)}$$  }
\end{lem}
\begin{proof}
Both identities are equivalent and follow by induction on $\ k.$
\end{proof}

\

\begin{exmp}
{\em For $\ k \geq 1, \ n \geq k+1, \ $ consider the numbers
$$T_{k,n} \ = \ \sum_{0<n_1 <  \cdots < n_k < n} n_1 \cdots n_k
\ \ \ \ \ \mbox{determined by the recursion}$$
$$T_{1,n}= \frac{n^2}{2}-\frac{n}{2} \ \ \ \ \mbox{and} \ \ \ \ T_{k+1,n}  =  \sum_{i=k+1}^{n-1}i\ T_{k,i}.$$
From Lemma \ref{l1} the continuous analogue for $\ T_{k,n} \ $  are given by  
$$\ \int_{\Delta^k_x}x_1\dots x_k dx_1\dots dx_k \ = \  
\frac{x^{2k}}{\prod_{i=1}^{k}\big(2i\big)} \ = \ \frac{(\frac{x^2}{2 \ })^k}{k!} . \ \ \ \ \ \  \mbox{Therefore}$$
$$\sum_{k=1}^{\infty}  \int_{\Delta^k_x}x_1\dots x_k dx_1\dots dx_k \ = \ 
\sum_{k=1}^{\infty}\frac{(\frac{x^2}{2})^k}{k!} \ = \ e^{\frac{x^2}{2\ }}-1 .$$
}
\end{exmp}

\

The following result follows from  Lemma \ref{l1}.\\

\begin{thm}
{\em Assuming sums  and integrals can be interchanged we have that
$$\int_{\Delta^k_x}\bigg( \sum_{a\in \mathbb{N}^k}f_a x^a \bigg)dx \  = \
\sum_{n=k}^{\infty}\bigg(\sum_{|a|=n-k} \frac{f_a}{\prod_{i=1}^{k}\big(|a|_i+i\big)} \bigg)x^n .$$
$$\int_{\Delta^k_x}\bigg( \sum_{a\in \mathbb{N}^k}f_a x^{(a)} \bigg)dx \  = \
\sum_{n=k}^{\infty}\bigg(\sum_{|a|=n-k} \prod_{i=2}^{k}\binom{|a|_i+i-1}{a_i}f_a  \bigg)  x^{(n)}.$$
}
\end{thm}

\

\begin{exmp}
{\em  Consider the series $\ \sum_{a\in \mathbb{N}}f_a x_i^a \ $  for  
$ \ 1 \leq i \leq k\ $ fixed. We have that
$$\int_{\Delta^k_x}\bigg(\sum_{a\in \mathbb{N}}f_a x_i^a\bigg)dx \  = \
\sum_{n=k}^{\infty}\frac{(i)^{(n-k)}}{n!}f_{n-k}x^n. \ \ \ \ \ 
\mbox{Similarly} $$ 
$$\int_{\Delta^k_x}\bigg(\sum_{a\in \mathbb{N}}f_a x_i^{(a)}\bigg)dx \  = \
\sum_{n=k}^{\infty}\binom{n-k+i-1}{i-1}f_{n-k}x^{(n)} .$$
In particular the $\ k\ $ antiderivative of $\ f(x)= \sum_{a\in \mathbb{N}}f_a x^{(a)} \ $ is given by
$$\int_{0}^{x}\cdots \int_{0}^{x_2}f(x_1)dx_1 \cdots dx_{k} \ = \
\int_{\Delta^k_x}\bigg(\sum_{a\in \mathbb{N}}f_a x_1^{(a)}\bigg)dx \  = \
\sum_{n=k}^{\infty}f_{n-k}x^{(n)} .$$

}
\end{exmp}

\

\begin{lem}\label{l5}
{\em For $\ k\in \mathbb{N}, \ a=(a_1,\dots,a_{k+1})\in \mathbb{N}^{k+1}, \ $ we have that
$$\int_{\Delta^k_x}x_1^{a_1}(x_2-x_1)^{a_2} \cdots (x-x_k)^{a_{k+1}}dx_1\dots dx_k \  = \ 
\frac{(a_1)!\cdots(a_{k+1})!}{\big(|a| + k\big)!}x^{|a| + k}  $$
$$\int_{\Delta^k_x}x_1^{(a_1)}(x_2-x_1)^{(a_2)} \cdots (x-x_k)^{(a_{k+1})}dx_1\dots dx_k \  =  \  x^{(|a|+k)} $$  }
\end{lem}
\begin{proof}
Both identities are equivalent and follow by induction on $\ k.$ 
\end{proof}

\

The following result follows from  Lemma \ref{l5}. \\

\begin{thm}
{\em Assuming sums and integrals can be interchanged we have that  
$$\int_{\Delta^k_x}\bigg( \sum_{a\in \mathbb{N}^{k+1}}f_a x_1^{a_1}(x_2-x_1)^{a_2} \cdots (x-x_k)^{a_{k+1}} \bigg)dx \  = \
\sum_{n=k}^{\infty}\bigg(  \sum_{|a|=n-k} \frac{(a_1)!\cdots(a_{k+1})!}{\big(|a| + k\big)!}f_a \bigg) x^n.$$
$$\int_{\Delta^k_x}\bigg( \sum_{a\in \mathbb{N}^{k+1}}f_a x_1^{(a_1)}(x_2-x_1)^{(a_2)} \cdots (x-x_k)^{(a_{k+1})} \bigg)dx \  = \
\sum_{n=k}^{\infty}\bigg(  \sum_{|a|=n-k} f_a \bigg)  x^{(n)}.$$ }
\end{thm}

\

\begin{exmp}
{\em For $\ k \geq 1, \ n \geq k+1, \ $ consider the numbers
$$U_{k,n} \ = \ \sum_{0<n_1 <  \cdots < n_k < n} n_1(n_2-n_1) \cdots
(n_k-n_{k-1})(n-n_k) \ \ \ \ \  \mbox{determined by }$$
$$U_{1,n}= \frac{(n-1)n(n+1)}{6} \ \ \ \ \ \mbox{and}\ \ \ \ \ U_{k+1,n}  =  \sum_{i=k+1}^{n-1}(n-i)T_{k,i}.$$
The continuous analogue for $\ U_{k,n} \ $ are the functions
$$V_{k}(x) \ = \ \int_{\Delta^k_x}x_1(x_2-x_1) \cdots
(x_k-x_{k-1})(x-x_k) dx_1\dots dx_k \ \ \ 
\mbox{determined by} $$
$$V_{1}(x)= \frac{x^3}{6} \ \ \ \ \  \mbox{and} \ \ \ \ \ V_{k+1}(x)  =  \int_{0}^{x}(x-t)V_{k}(t)dt.$$
The functions $\ \displaystyle V_{k}(x)= \frac{x^{2k+1}}{(2k+1)!} \ $ solve the desired recursion.
}
\end{exmp}

\

\section{Lattice paths \ vs \ directed paths}

In this section we review some known facts on Young diagrams and numerical partitions, and introduce directed paths as continuous analogue for lattice paths. A Young diagram is a graphical representation for a partition $\ \lambda=(\lambda_{k},\cdots, \lambda_1)\in \mathbb{N}_{>0}^{k}\ $ of 
$\ a=\lambda_{k}+ \cdots + \lambda_1 \ $ where 
$\ \lambda_{k} \geq \cdots \geq \lambda_1 > 0. \ $ 
We assume that the southwesternmost point of a Young diagram $\ \lambda \ $ is placed at
the origin $ \ (0,0). \ $ The height $\ h_{\lambda}\ $ of  $\ \lambda \ $ is the number $ \ k\ $ of blocks, the width $\ w_{\lambda}\ $ of  $\ \lambda \ $ is equal to  $\ \lambda_{k}, \ $ the length of the longest block. The boundary of the underlying region of $\ \lambda \ $ is given by the vertical segment from $\ (0,0) \ $ to $\ (0,h_{\lambda}), \ $ the horizontal segment from $\ (0,h_{\lambda}) \ $ to
$\ (w_{\lambda}, h_{\lambda}), \ $ and the east-north lattice path $\ p_{\lambda}\ $ from $\ (0,0) \ $ to $\ (w_{\lambda}, h_{\lambda}). \ $ \\

Figure \ref{fig:yd}-a)  displays the Young diagram of the partition 
$\ \lambda=(6,5,3,3,1)\ $ of $\ 18 \ $ with $\ h_{\lambda}=5,\ $ and $\ w_{\lambda}=6.\ $   Thus $\ 18 \ $ is the area of the region bounded by
the segment  from $\ (0,0) \ $ to $\ (0,5), \ $ 
the segment from $\ (0,5) \ $ to
$\ (6, 5), \ $ and the $\ 11  $-steps east-north lattice path $\ p_{\lambda} \in \mathrm{L}(6,5,18):$
$$\big((0,0),(1,0),(1,1),(2,1),(3,1), (3,2),(3,3),(4,3),(5,3),(5,4),(6,4),(6,5)\big).$$
 
\begin{figure}
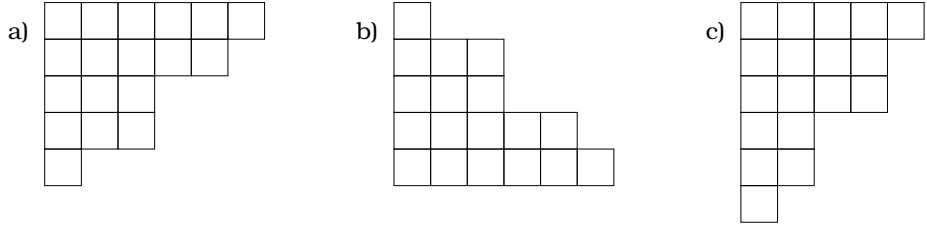

  $$\scalebox{0.9}{a) \ \ydiagram{6,5,3,3,1}\ \ \ \ \ \ \ \ \ \ \ \
 b) \ \ydiagram{1,3,3,5,6}
\ \ \ \ \ \ \ \ \ \ \ \    c) \ \ydiagram{5,4,4,2,2,1}}$$
  \caption{a) Young diagram $\lambda.$ \  b) Stair representation of  $\lambda.$  \ c) Conjugated Young diagram $\lambda^t.$ }
  \label{fig:yd}
\end{figure}
\noindent Figure \ref{fig:yd}-b) displays a Young diagram using the stair convention. We learn from this representation that Young diagrams can be thought as stairs built up placing square bricks on top of each other in a left-aligned fashion.\\

For $\ n,m \geq 1 \ $ let $\ \mathrm{Y}(m,n) \ $ be the set of Young diagrams with $\ n\ $ blocks such that each of its longest blocks has  $\ m \ $ boxes. Let $\ \mathrm{Y}(m,n,a)\ $ be the set of Young diagrams in $\ \mathrm{Y}(m,n) \ $ with  $\ a\ $ boxes. Let  
$\ \mathrm{L}(m,n) \ $ (\  $\mathrm{L}(m,n,a) \ $) be the set of lattice east-north paths  (\ enclosing a region of area $\ a\ $)  that  star at $\ (0,0),\ $ end at $\ (m,n),\ $ proceed by  steps of length $\ 1 \ $ either in the east or north direction, starting with an east direction step, and ending with an north direction step. Paths in  $\ \mathrm{L}(n,m) \ $ have length $\ m+n.\ $ The map sending a Young diagram $\ \lambda \ $ to its associated east-north lattice path $\ p_{\lambda} \ $ establishes a bijective correspondence between $\ \mathrm{Y}(m,n) \ $ and $\ \mathrm{L}(m,n), \ $ as well as between $\ \mathrm{Y}(m,n,a) \ $ and $\ \mathrm{L}(m,n,a).$ \\

Note that $ \  \big|\mathrm{Y}(m,n,a)\big|   \  $  is the cardinality of the set of numerical partitions of $\ a \ $ with $\ n \ $ summands such that each largest summand is equal to $ \ m, \ $ thus
$$ \ \big|\mathrm{Y}(m,n)\big|= \sum_{a=m+n-1}^{nm}\big|\mathrm{Y}(m,n,a)\big| . \ $$ Since a Young diagram $\ \lambda \ $ and its dual $\ \lambda^t, \ $
see Figure \ref{fig:yd}-c), have the same underlying area, and 
$\ h_{\lambda}= w_{\lambda^t}, \  w_{\lambda}= h_{\lambda^t}, \ $
we have that   
$\ \ \big|\mathrm{Y}(m,n,a)\big|=\big|\mathrm{Y}(n,m,a)\big| \ \ \  \mbox{and} \ \ \   \big|\mathrm{Y}(m,n)\big|= \big|\mathrm{Y}(n,m)\big|.$

\

To introduce our continuous analogues for Young diagrams we need the notion of  east-north directed paths, which is obtained from that of east-north lattice paths by lifting the restriction of length $\ 1 \ $ steps, i.e. directed paths are allowed to take steps of arbitrary length, including steps of length zero, but are constrained to move either in the east or north direction in alternating fashion. \\

For $\ (x,y)\in \mathbb{R}_{> 0}^2\ $ and $\ n\geq 1\ $ let $\ \mathrm{P}_n(x,y) \ $ be the set of $\ 2n \ $ steps east-north directed paths from $\ (0,0) \ $ to $\ (x,y).\ $
Thus a path $\ p\in \mathrm{P}_n(x,y) \ $ is determined by a $\ 2n+1\ $ tuple of points $\ p=(p_0,...,p_{2n}) \in \mathbb{R}_{\geq 0}^2\ $ such that:
\begin{itemize}
  \item $p_0=(0,0)\ $ and $\ p_{2n}=(x,y).$
  \item $p_{2i+1}=p_{2i}+(t,0)\ $ for some $\ t\in \mathbb{R}_{\geq 0}.$
  \item $p_{2i}=p_{2i-1}+(0,t)\ $ for some $\ t\in \mathbb{R}_{\geq 0}.$
\end{itemize}

\

\begin{rem}
{\em  The Young diagram from Figure \ref{fig:yd} is bounded by an $\ 11 \ $ steps lattice path, however, it is bounded by an  $\ 8 \ $ steps directed path, namely
$$(0,0)\rightarrow (1,0)\rightarrow (1,1)\rightarrow (3,1)\rightarrow
(3,3)\rightarrow (5,3)\rightarrow (5,4)\rightarrow (6,4)\rightarrow (6,5).$$  
}
\end{rem}

\

\noindent Let $\ \pi_x \ $ and $\ \pi_y \ $ be the plane projections onto the $\ x$-axis and $\ y$-axis, respectively. We introduced Cartesian coordinates on $\ \mathrm{P}_n(x,y) \ $ by noticing that there is a bijective correspondence $$\mathrm{P}_n(x,y) \ \ \longleftrightarrow  \ \ 
\Delta^{n-1}_x \times \Delta^{n-1}_y \ \ \ \ \  \mbox{given by}$$
\begin{itemize}
  \item $ (p_0,...,p_{2n})  \in \mathrm{P}_n(x,y) \ $ is send to  $\ \ ((x_1,...,x_{n-1})\ ; \ (y_1,...,y_{n-1})) \in \Delta^{n-1}_x \times \Delta^{n-1}_y \ \ $ where
      $ \ \ x_{i}=\pi_x(p_{2i-1})  \ \  \mbox{and}  \ \ 
      y_{i}=\pi_y(p_{2i})\ \ $ for $\ \ 1\leq i \leq n-1.$
      
  \item $((x_1,...,x_{n-1})\ ; \ (y_1,...,y_{n-1})) \in \Delta^{n-1}_x \times \Delta^{n-1}_y \ $ is send to $ \ (p_0,...,p_{2n})  \in \mathrm{P}_n(x,y) \ $ where the points $\ p_i \ $ are given  for $\ 1 \leq i \leq n\ $ by
      $\ \   p_{2i-1}=(x_i,y_{i-1})  \ \  \mbox{and} \ \  p_{2i}=(x_i,y_i),\ \ $ where  $\ x_0=0, \ x_n=x,\  y_0=0, \  y_n=y. \ $
\end{itemize}

\

Thus we learn that  $\ \mathrm{P}_n(x,y) \ $ 
may be identified with a top dimensional convex polytope in 
$\ \mathbb{R}^{n-1}\times \mathbb{R}^{n-1}, \ $ and thus inherits differential structures; for example,  we get that
$$\mathrm{vol}(\mathrm{P}_n(x,y)) \ = \ \mathrm{vol}(\Delta^{n-1}_x \times \Delta^{n-1}_y)  \ = \ \mathrm{vol}(\Delta^{n-1}_x) \mathrm{vol}(\Delta^{n-1}_y)\ = \ \frac{x^{n-1}}{(n-1)!}\frac{y^{n-1}}{(n-1)!}.$$ 
Integrals over $\ \mathrm{P}_n(x,y) \ $ can be computed recursively as follows
$$\int_{\mathrm{P}_n(x,y)}f(x_1,...x_{n-1}\ ; \ y_1,...y_{n-1})dx_1\cdots d_{y_{n-1}} \ = $$
$$ \int_{0}^{x}\int_{0}^{y}\bigg( \int_{\mathrm{P}_{n-1}(x_{n-1},y_{n-1})}f(x_1,...x_{n-1}\ ; \ y_1,...y_{n-1})dx_1\cdots d_{y_{n-2}} \bigg)dx_{n-1}dy_{n-1}.$$

\

\section{Continuous diagrams}\label{s3}

For easy reading of this section is good to have  Figure \ref{fig:cyd}, Figure \ref{fig:dd}, and  Example \ref{2} in mind.
We are ready to define a continuous analogue for Young diagrams.
The diagram $\ D_p \ $ associated to a directed path $\ p \in \mathrm{P}_n(x,y)\ $ is  the planar region  enclosed by the interval from $\ (0,0)\ $ to $\ (0,y),\ $ the interval from $\ (0,y)\ $ to  $\ (x,y),\ $ and the directed path $ \ p, \ $ together with the tessellation of  $\ D_p \ $ by the rectangles $\ D_{p,i} \ $  with vertices $\ (0,y_{i-1}), \ (0,y_{i}), \ (x_i,y_{i-1}), \ (x_i,y_{i}) \ $ for
$\ 1 \leq i \leq n. \ $  In other words $\ (s,t) \in  D_{p,i} \ $ if and only if
$\ 0 \leq s \leq x_i \ $ and $\  y_{i-1} \leq t \leq y_{i-1}. $  
Therefore, $\ (s,t) \in \ D_{p} \ $ if and only if
$\  y_{i-1} \leq t \leq y_{i}\ $ and $\ 0 \leq s \leq x_i \ $ for some $\ 1 \leq i \leq n. \ $
The width and height of $\ D_p \ $ are given by 
 $\ w(D_p)=x \ $ and   $ \ h(D_p)=y,\ $ respectively.\\

\begin{figure}
  \begin{tikzpicture}
\draw  (0,0) -- (1.3,0)   (1.3,0) -- (1.3,0.5)
   (1.3,0.5) -- (1.5,0.5) (1.5,0.5) -- (1.5,0.7)
    (1.5,0.7)--(1.9,0.7)  (1.9,0.7)--(1.9,2.3)  
   (1.9,2.3) -- (3.1,2.3) (3.1,2.3) -- (3.1,2.5) 
   (0,0) -- (0,2.5) (0,2.5) -- (3.1,2.5)

   (4,0) -- (5.3,0)  (5.3,0) -- (5.3,0.5)
   (4,0.5) -- (5.3,0.5) (5.3,0.5)--(5.5,0.5) (5.5,0.5) -- (5.5,0.7)
   (4,0.7) -- (5.5,0.7) (5.5,0.7)--(5.9,0.7)
   (5.9,0.7)--(5.9,2.3)  (4,2.3) -- (5.9,2.3)
   (5.9,2.3) -- (7.1,2.3) (7.1,2.3) -- (7.1,2.5) (4,0) -- (4,2.5)
   (4,2.5) -- (7.1,2.5)
   
   (8,0) -- (9.3,0)   (9.3,0) -- (9.3,2.5)
    (9.3,0.5)--(9.5,0.5) (9.5,0.5) -- (9.5,2.5)
    (9.5,0.7)--(9.9,0.7)
   (9.9,0.7)--(9.9,2.5) 
   (9.9,2.3) -- (11.1,2.3) (11.1,2.3) -- (11.1,2.5) (8,0) -- (8,2.5)
   (8,2.5) -- (11.1,2.5)
   
   ;
\end{tikzpicture}

  \caption{  Underlying region of $\ D_p \ $ on the left,  diagram  $\ D_p \ $ in the middle,  vertical tessellation on the right. }
  \label{fig:cyd}
\end{figure}
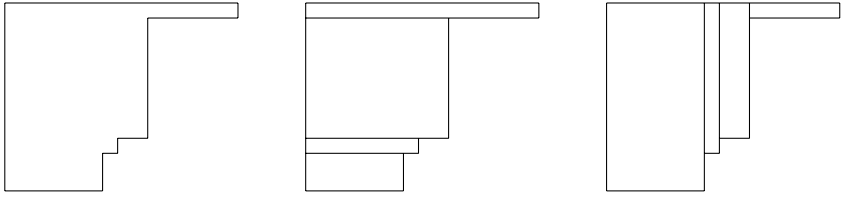

The area of the rectangle $\ D_{p,i}\ $ is given by 
$\ \ \mathrm{area}(D_{p,i})=x_i(y_{i}-y_{i-1}). \ \ $
Therefore the undelying area of $\ D_p \ $ is given by
$$\mathrm{area}(D_{p})\ = \ \sum_{i=1}^{n}x_i(y_i-y_{i-1}).$$
The sum above represents a continuous analogue for a decomposition of $\ \mathrm{area}(D_{p}). \ $ These sort of decompositions come with extra structure since each summand is itself a product as we are keeping track of the length $\ x_i \ $ and the height $\ y_i-y_{i-1}\ $ of the rectangles $\ D_{p,i},\ $ while in the combinatorial setting for numerical partitions only lengths are allowed to vary while heights are fixed to $\ 1. \ $ \\

And alternative tessellation for the underlying region of $\ D_{p}\ $ is obtained by considering the rectangles $\ D^v_{p,i}\ $ with vertices $\ (x_{i-1},y_{i-1}), 
\ (x_i,y_{i-1}), \  (x_{i-1},y), \ (x_i,y),\ $ for $ 1\leq i \leq n, \ $ thus $\ (s,t)\in D^v_{p,i} \ $ if and only if $\ x_{i-1} \leq s \leq x_i \ $ and $\  y_{i-1} \leq t \leq y.\ $ The vertical tessellation leads 
$$\mathrm{area}(D_{p})\ = \  \sum_{i=1}^{n}\mathrm{area}(D^v_{p,i}) \ = \ \sum_{i=1}^{n}(x_i-x_{i-1})(y-y_{i-1}).$$ 

\

Just as in the combinatorial setting, there is a stair diagram
$\ D_p^s \ $ for each diagram $\ D_p \ $ with $\ p \in \mathrm{P}_n(x,y).\ $ The diagram $\ D_p^s \ $ represents a stair built in a left justified fashion placing rectangular bricks on top of each other. The underlying region of $\ D_p^s\ $ is enclosed by the interval $\ (0,0)\ $ to $\ (0,y),\ $ the interval from $\ (0,0)\ $ to  $\ (x,0),\ $ and the east-south directed path $ \ q= (q_0,\cdots, q_{2n})\ $ where $\ \ q_i=(\pi_x(p_i),y-\pi_y(p_i)). \ \ $ The underlying region of $\ D_p^s \ $ comes with the tessellation by the rectangles with vertices $\ (0,\pi_y(q_{2i-1})),  \ (0,\pi_y(q_{2i})), \ q_{2i-1}, \ q_{2i} \ $ for $\ 1 \leq i \leq n. \ $ Clearly the underlying regions of $\ D_p \ $ and  $\ D_p^s \ $  have the same area.\\

For $\ p \in \mathrm{P}_n(x,y)\ $ the dual diagram $\ D_p^t=D_{p^t} \ $ of  $\ Y_p \ $ is the diagram associated to the directed east-north path $\ p^t \in \mathrm{P}_n(y,x) \ $ given in Cartesian coordinates by 
$$\ p^t=(x_1^t,...,x_{n-1}^t\ ; \ y_1^t,...,y_{n-1}^t) \in 
\Delta^{n-1}_y \times \Delta^{n-1}_x \ \ \ \ \ \mbox{where} \ \ \ \ \ 
x_i^t=y-y_{n-i} \ \ \ \mbox{and} \ \ \  y_i^t=x-x_{n-i}.$$
The underlying area of $\  D_p^t \ $ is equal to the underlying area of $\ D_{p}, \ $ indeed
$$\mathrm{area}(D_{p^t}) \ = \ \sum_{i=1}^{n}x_i^t(y_i^t-y_{i-1}^t) \ = \
\sum_{i=1}^{n}(y-y_{n-i})(x_{n-i+1}-x_{n-i})\ = \ \mathrm{area}(D_{p}),$$
where the latter identity follows from the vertical tessellation of 
$\ D_{p}.  $ \\

\begin{figure}
  \begin{tikzpicture}
  \draw  (0,0)--(3.1,0) (0,0)--(0,2.5) (0,0.2)--(3.1,0.2)
   (3.1,0)--(3.1,0.2) (0,1.8)--(1.9,1.8)
   (1.9,0.2)--(1.9,1.8) (1.5,1.8)--(1.5,2)
   (0,2)--(1.5,2) (0,2.5)--(1.3,2.5)(1.3,2)--(1.3,2.5)
  
   (5,0)--(5.2,0) (5,0)--(5,3.1) (5,1.2)--(6.6,1.2)
   (5.2,0)--(5.2,1.2) (5,1.6)--(6.8,1.6) (6.6,1.2)--(6.6,1.6)
   (5,1.8)--(7.5,1.8) (6.8,1.6)--(6.8,1.8) (5,3.1)--(7.5,3.1) 
   (7.5,1.8)--(7.5,3.1);
\end{tikzpicture}

  \caption{Stair diagram  $\ D_p^s \ $ on the left, and dual diagram $\ D_p^{t}\ $ on the right.}
  \label{fig:dd}
\end{figure}
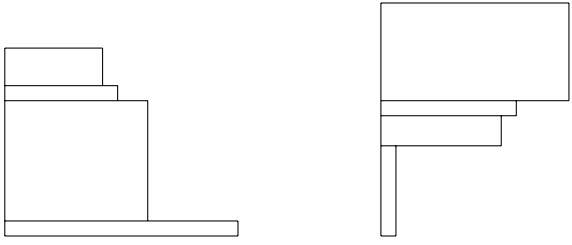

\begin{thm}{\em The dual map 
$\ \ (\ )^t:\mathrm{D}_n(x,y) \longrightarrow \mathrm{D}_n(y,x) \ \ $ is a volume preserving smooth involution. It preserves orientation if and only if $\ n \ $ is odd.
}
\end{thm}

\begin{proof} Smoothness and involution properties follow from the fact that
the map $$( \ )^t:\Delta^{n-1}_x \times \Delta^{n-1}_y \longrightarrow \Delta^{n-1}_y \times \Delta^{n-1}_x\ \ \ \ \ \mbox{is given by} $$
$$ (x_1,...,x_{n-1}\ ; \ y_1,...,y_{n-1})^t \ = \ (y-y_{n-1},...,y-y_{1}\ ; \ x-x_{n-1},...,x-x_{1}).$$
The volume and orientation properties follow from the identities
$$dx_1^t\wedge...\wedge dx_{n-1}^t \wedge dy_1^t\wedge...\wedge dy_{n-1}^t \ = $$  
$$(-1)^{2(n-1)}dy_{n-1}\wedge...\wedge dy_{1} \wedge dx_{n-1}\wedge...\wedge dx_{1} \ =$$
$$(-1)^{\frac{(n-2)(n-1)}{2}+\frac{(n-2)(n-1)}{2}}dy_{1}\wedge...\wedge dy_{n-1} \wedge dx_{1}\wedge...\wedge dx_{n-1} \ = $$
$$(-1)^{(n-1)^2}dx_{1}\wedge...\wedge dx_{n-1} \wedge dy_{1}\wedge...\wedge dy_{n-1}.$$
\end{proof}

\begin{exmp}\label{2}
{\em Let $\ D_p \in \mathrm{D}_4(3.1,2.5) \ $ be the diagram with east-north directed path $\ p\ $ given by
$$ (0,0) \rightarrow (1.3,0)  \rightarrow (1.3,0.5) \rightarrow (1.5,0.5)
\rightarrow (1.5,0.7)\rightarrow$$ 
$$ (1.9,0.7)\rightarrow (1.9,2.3)
\rightarrow (3.1,2.3) \rightarrow (3.1,2.5).$$
In Cartesian coordinates $\ \ D_p= (1.3,1.5,1.9\ ; \ 0.5,0.7,2.3) \in \Delta_{3.1}^3\times \Delta_{2.5}^3. \ \ $ 
The area of $\ D_p\ $ is $\ 4.61 .$\\

\noindent Figure \ref{fig:cyd} \ displays  the underlying region of $\ D_p \ $
on the left,  diagram $\ D_p\ $ in the middle, and the vertical tessellation for $\ D_p \ $ on the right.  Figure \ref{fig:dd}  \ displays  the stair representation $\ D_p^s \ $ on the left, and  the dual diagram $\ D_p^{t}\ $ 
of the diagram $\ D_p \ $ on the right.\\

\noindent The stair diagram of $\ D_p^s \ $ is given by the east-south directed path
$$(0,2.5) \rightarrow (1.3,2.5) \rightarrow (1.3,2) \rightarrow (1.5,2) \rightarrow (1.5,1.8) \rightarrow$$  
$$  (1.9,1.8) \rightarrow(1.9,0.2)\rightarrow
(3.1,0.2) \rightarrow (3.1,0).$$
The dual diagram $\ D_p^{t} \ $ is determined by the east-north directed path
$$(0,0) \rightarrow (0.2,0) \rightarrow (0.2,1.2) \rightarrow (1.8,1.2) \rightarrow (1.8,1.6) \rightarrow$$  
$$  (2,1.6) \rightarrow(2,1.8)\rightarrow
(2.5,1.8) \rightarrow (2.5,3.1).$$
with Cartesian coordinates $\ \  (0.2,1.8, 2 \ ; \ 1.2,1.6,1.8  ) \in \Delta_{2.5}^3\times \Delta_{3.1}^3. \ $

}
\end{exmp}

\

Let $\ \mathrm{P}\ = \ \bigsqcup_{ x\geq 0, y \geq 0}\mathrm{P}(x,y)\ = \ \bigsqcup_{x\geq 0, y \geq 0} 
\bigg( \bigsqcup_{n\geq 0}\mathrm{P}_n(x,y) \bigg)\ = \
\bigsqcup_{n\geq 0}  
\bigg( \bigsqcup_{x\geq 0, y \geq 0} \mathrm{P}_n(x,y) \bigg) \ $
be the set of all east-north paths starting at $\ (0,0). \ $ Note that
we are formally introducing a path $\ \diamond \ $ of length zero from $\ (0,0)\ $ to $\ (0,0),\ $ and its corresponding diagram, and that there are no other paths of length zero. Set
$$\mathrm{D}_n(x,y) = \{ D_p \ | \ p \in \mathrm{P}_n(x,y) \} ,\ \ 
 \ \mathrm{D}(x,y) = \{ D_p \ | \ p \in \mathrm{P}(x,y) \} ,\ \ \ 
 \mathrm{D} = \{ D_p \ | \ p \in \mathrm{P} \} .$$
We introduce differential and geometric structures on 
$\ \mathrm{D}_n(x,y) \subseteq  \mathrm{D}(x,y)\subseteq \mathrm{D}, $ from the corresponding structures on $\  \mathrm{P}_n(x,y) \subseteq  \mathrm{P}(x,y)\subseteq \mathrm{P}. $ \\

\begin{prop}
{\em  The following properties hold. 
\begin{description}
  \item[a)] $\ \mathrm{D} \ $ adquires a $\ \mathbb{N}$-graded smooth convex monoid structure with unit $\ \diamond \ $ and concatenation product 
$ \  D_p \ast D_q \ = \ D_{p \ast q} , \  $ where
$$(p_0, \cdots, p_n)\ast (q_0, \cdots, q_m) \ = \
  (p_0, \cdots, p_n, p_n+q_1, \cdots, p_n+q_m).$$
  \item[b)] The map $\ \mathrm{area}: \mathrm{D} \longrightarrow \mathbb{R}_{\geq 0}\  $ satisfies
$$ \mathrm{area}(D_p \ast D_q ) \ = \  \mathrm{area}(D_p) \ + \ \mathrm{area}(D_q) \ + \ xz, $$ for $\ \ D_p \in \mathrm{D}_n(x,y), \ \ D_q \in \mathrm{D}_m(w,z)  .$
  \item[c)] Concatenation $ \ \ast: \mathrm{D}_n(x,y) \times \mathrm{D}_m(w,z) \longrightarrow 
\mathrm{D}_{n+m}(x+w,y+z) \ $ is given by
 $$(x_1,...,x_{n-1}\ ; \ y_1,...,y_{n-1})\ast (w_1,...,w_{m-1}\ ; \ z_1,...,z_{m-1}) \ = $$
 $$(x_1,...,x_{n-1},x,x+w_1,...,x+w_{m-1} \ ; \
 y_1,...,y_{n-1},y,y+z_1,...,y+z_{m-1}) .$$
 \item[d)] Concatenation $ \ \ast: \mathrm{D}_n(x,y) \times \mathrm{D}_m(w,z) \longrightarrow 
\mathrm{D}_{n+m}(x+w,y+z) \ $ is a convex map:
 $$\big(tD_p+(1-t)D_{p'}\big)\ast \big(tD_q+(1-t)D_{q'}\big) \ = \ tD_{p\ast q } \ + \ 
 (1-t)D_{p'\ast q'}\ \ \ \ \mbox{for} \ \ \ t\in [0,1].$$
\end{description}
}
\end{prop}

\begin{proof}
The unit axiom, associativity, and item c) follow since
the product $\ \ast \ $ of east-north directed paths is given concatenation of paths after translating the second path so that it begins where the first path ends. For item b) note that the underlying region of 
$ \ D_{p \ast q} \ $ can be subdivided in tree pieces, the first one is the underlying region of $ \ D_{p}, \ $ the second one is the underlying region of $ \ D_{q} \ $ translated by $\ (x,y), \ $  and the third one is the rectangle of area $\ xz \ $ with vertices 
$\ (0,y),  \ (x,y),  \ (x,y+z),  \ (0,y+z). \ $ Item d) follows from item c).
\end{proof}

\

Next we introduce a partial order on the space of continuous diagrams
$\  \mathrm{D}. \  $ Intuitively $\ D_p \leq D_q \ $ if the underlying region of
$\ D_p, \ $ translated upwards so that its upper left corner  
matches the upper left corner of $ \ D_q, \ $ fits inside the underlying region of $ \ D_q. \ $  Path $\ \diamond \ $ is the minimum of $\  \mathrm{D}. \  $

\

\begin{prop}\label{k}
{\em  The following properties hold.
\begin{description} 
  \item[a)] $ \mathrm{D} \  $ admits a partial order where  for $\ D_p \in \mathrm{D}_n(x,y)\ $ and $ \ D_q \in \mathrm{D}_m(w,z) \ $ we set 
  $$\ D_p \leq D_q \ \ \ \ \ \mbox{ if and only if } 
  \ \ \ \ \ y \leq z \ \ \ \ \mbox{and} \ \ \ \ D_p + (0,z-y) \ \subseteq \  D_q. \ $$
  Note that $ \ (s,t) \in D_p + (0,z-y) \ $ if and only if
 there is  $\ 1 \leq i \leq n \ $ such that 
  $$\  y_{i-1}+z-y \leq t \leq y_{i}+z-y\ \ \ \mbox{and} \ \ \ 
    0 \leq s \leq x_i.$$
    
  \item[b)] $D_p \leq D_q \ $ if and only if $\ (x_i,y_{i-1}+z-y) \in D_q  \ $
  for $1 \leq i \leq n, \ $ i.e.  $\ D_p \leq D_q \ $ if and only if
  for each $\ \ 1 \leq i \leq n \ $ there is $ \ 1 \leq j \leq m \  $ such that
  $\ \  z_{j-1} \leq y_{i-1}+z-y \leq z_j\ \ \ \mbox{and}
    \ \ \ 0 \leq x_i \leq w_j.$ 

  \item[c)] $ D_p \leq D_q \ $ implies that 
  $\ \ \displaystyle \sum_{i=1}^{n}x_i(y_i-y_{i-1}) \ = \ \mathrm{area}(D_p)\ \leq \ \mathrm{area}(D_q) \ = \ \sum_{i=1}^{m}w_i(z_i-z_{i-1}).$ 

\end{description}
}
\end{prop}

\begin{proof} 
 a) Reflexivity holds as $ \ y \leq y \ $ and $\ D_p \subseteq D_p. \ $ 
 Anti-symmetry holds as we would have $ \ y \leq z \leq y \ $ and 
 $\ D_p \subseteq D_q \subseteq D_p. \ $ For transitivity note that
 if $ \ y\leq z\leq w, \ $   $ \ D_p + (0,z-y)\subseteq D_q, \ $ and
 $\ D_q + (0,w-z) \subseteq D_r, \ $ then
 $ \ \ D_p + (0,w-y)\ = \ \big(D_p + (0,z-y)\big)+(w-z)\ \subseteq \ D_q +(w-z) \ \subseteq \ D_r.\ \ $
  The last part of item a) follows from the characterization of the points in the underlying region of a diagram given in the first paragraph of section \ref{s3}, taking upwards translations into account.  To check that
 $\ D_p + (0,z-y) \ \subseteq \  D_q \ $ it is enough to check that the vertical translation by $\ (0,z-y) \ $ of the points $ \ p_{2i-1} \ $ of the path $\ p\ $ are in   $\  D_q, \ $ thus item b) holds. Item c) follows from the invariance of area under translations.
\end{proof}

\

\begin{prop}
{\em  Consider the posets $\ (\mathrm{D}_n(x,y), \leq) \ $ with the partial order $\ \leq \ $ coming from  Proposition \ref{k}. For $\ r \in \mathrm{D}_l(w,x), \
 \ q \in \mathrm{D}_m(y,z) \ $ the maps
$$(\ )\ast q: \mathrm{D}_n(x,y) \rightarrow \mathrm{D}_{n+m}(x,z)
\ \ \ \ \mbox{and} \ \ \ \ r\ast (\ ): \mathrm{D}_n(x,y) \rightarrow \mathrm{D}_{l+n}(w,y)$$
are order preserving.
}
\end{prop}

\begin{proof}
The partial order on $ \ \mathrm{D}_n(x,y) \ $  no longer  uses translations, so $\ D_p  \leq   D_q \ $ if and only if 
$\ D_p \subseteq   D_q, \ $ thus $\ D_p  \leq   D_q \ $ if and only if the path $\ p \ $ is on the left side of the path $\ q. \ $ Using either of these characterization, it is clear that the statement of the proposition holds.
\end{proof}

\

\section{Volume of the space of diagrams}

Next we use the continuous diagrams introduced in the last section to build continuous analogues for some quantities from the theory of Young diagrams. A diagram in $ \  \mathrm{YD}(m,n) \ $ has height (and thus number of blocks)  $\ n \ $ and  width  $\ m. \ $ 
On the continuous side height and number of blocks need not be equal, so we are led to consider the spaces $\ \mathrm{D}_n(x,y) \ $ with fixed width $\ x, \ $ height $\ y,\ $ and number of blocks $\ n. \ $ We have that
$$ \mathrm{vol}(\mathrm{D}_n(x,y))\ =  \ 
\frac{x^{n-1}}{(n-1)!}\frac{y^{n-1}}{(n-1)!}.$$
Note that  
$ \ \ \mathrm{vol}(\mathrm{D}_{n+1}(x,y)) < \mathrm{vol}(\mathrm{D}_{n}(x,y))\ \ $  for $\ \ n > \sqrt{xy}. \ $\\

Below we compute the volume of a disjoint union of manifolds as the sum of the volume of each piece, even if the pieces have different dimensions. 
Our continuous analogue for the number $ \ |\mathrm{YD}(m,n)\big| \ $ of Young diagrams with width $\ m \ $ and height $\ n \ $ is given by
$$ \ \mathrm{vol}(\mathrm{D}(x,y))\ =  \ \sum_{n=1}^{\infty}\mathrm{vol}(\mathrm{D}_n(x,y)).$$
Note that height is fixed but the number of blocks is actually allowed to vary
from $\ 1 \ $ to infinity. \\

Recall that the modified Bessel functions $\ I_{\nu} \ $ are given 
for $\ \nu \in \mathbb{N} \ $ by
$$\displaystyle I_{\nu}\left(z\right)\ = \ (\tfrac{z}{2})^{\nu}\sum_{k=0}^{\infty}\frac{(\frac{z}
{2})^{2k}}{k!\left(\nu+k\right)!}.$$
Next couple of results follow from the definitions of 
$\ \mathrm{vol}(\mathrm{D}(x,y)) \ $ and $\ I_{\nu}\left(z\right).$\\

\begin{thm}\label{t2}
{\em For $\ x>0, \ y> 0 \ $ the following properties hold.
\begin{description}

   \item[a)] $\displaystyle \mathrm{vol}(\mathrm{D}(x,y))\ = \  \sum_{n=0}^{\infty}\frac{x^ny^n}{n!n!}.$
  \item[b)] $\mathrm{vol}(\mathrm{D}(x,y))\ = \ \mathrm{vol}(\mathrm{D}(x,y)).\ $
  \item[c)]$ \mathrm{vol}(\mathrm{D}(ax,a^{-1}y)) \ = \  \mathrm{vol}(\mathrm{D}(x,y)) \ $ for all $\ a\in \mathbb{R}_{>0}.$
  \item[d)] $\displaystyle \bigg(x\frac{\partial}{\partial x}-y\frac{\partial}{\partial y}\bigg) \mathrm{vol}(\mathrm{D}(x,y))\ = \ 0. $
  \item[e)] $\displaystyle  \bigg( \frac{\partial}{\partial x}\frac{\partial}{\partial y}\ - \ 1\bigg) \mathrm{vol}(\mathrm{D}(x,y))\ = \ 0.$
  \item[f)] $ \mathrm{vol}(\mathrm{D}(x,y)) \ = \  I_0(2\sqrt{xy}) \ \ \ \ \ \ 
  \mbox{and}\ \ \ \ \ \ I_0(z)\ = \ \mathrm{vol}(\mathrm{D}(\frac{z}{2},\frac{z}{2})). $
\end{description}
}
\end{thm}

\

From Theorem \ref{t2}  we see that $ \ \mathrm{vol}(\mathrm{D}(x,y)) \ $ can be extended to a symmetric entire analytic function, it satisfies an hyperbolic second-order linear partial differential equation with constant coefficients, as well as a first-order linear partial differential equation.\\

\begin{thm}\label{t3}
{\em  For $\ x>0, \ y> 0, \ z>0 \ $ the following properties hold.
\begin{description}

  \item[a)] $  \displaystyle \int_{\Delta^k_x}\mathrm{vol}(\mathrm{D}(x,y))\ dx \ = \   \frac{\partial^k}{\partial y^k}\mathrm{vol}(\mathrm{D}(x,y)) \ = \
      \sum_{n=0}^{\infty}\frac{x^{n+k}y^{n}}{(n+k)!n!} .$
      
  \item[b)] $  \displaystyle
      \int_{\Delta^k_y} \mathrm{vol}(\mathrm{D}(x,y))\ dy\  = \ \frac{\partial^k}{\partial x^k}\mathrm{vol}(\mathrm{D}(x,y)) \ = \ \sum_{n=0}^{\infty}\frac{x^ny^{n+k}}{n!(n+k)!}. $

\item[c)] $ \displaystyle \frac{\partial^k}{\partial x^k}\mathrm{vol}(\mathrm{D}(x,y)) \ = \ \bigg(\frac{y}{x}\bigg)^{\frac{k}{2}} I_k(2\sqrt{xy}). $
    
    \item[d)] $ \displaystyle 
     \frac{\partial^k}{\partial y^k}\mathrm{vol}(\mathrm{D}(x,y)) \ = \ \bigg(\frac{x}{y}\bigg)^{\frac{k}{2}} I_k(2\sqrt{xy}). $

\item[e)]$\displaystyle I_k(z) \  = \ \frac{\partial^k\rho}{\partial x^k}\bigg(\frac{z}{2},\frac{z}{2}\bigg) \ = \ \frac{\partial^k\rho}{\partial y^k}\bigg(\frac{z}{2},\frac{z}{2}\bigg) , \ \ $ where $\ \ \rho(x,y)  =  \mathrm{vol}(\mathrm{D}(x,y)).\ $ 
\end{description}
}
\end{thm}

\

From Theorem \ref{t3}-a),b),c),d) \ we see that $ \ \mathrm{vol}(\mathrm{D}(x,y)), \  $ its derivatives, and its antiderivatives can be written in terms of the modified Bessel functions $\ I_{\nu} \ $ for $\ \nu \in \mathbb{N} ,\ $ conversely, from Theorem \ref{t3}-e) \ we see that the modified Bessel function
$\ I_{\nu}, \ $ for $\ \nu \in \mathbb{N}, \ $  can be written in terms of $ \ \mathrm{vol}(\mathrm{D}(x,y)) \ $ and its derivatives.\\

Let $\ \mathrm{D}^{1}(x,y) \ $ be the set of continuous diagrams
with width less than or equal to  $\ x \ $ and height $\ y ; \ $
let $\ \mathrm{D}^{2}(x,y) \ $ be the set of continuous diagrams
with width $\ x \ $ and height less than or equal to $\ y ; \ $
let $\ \mathrm{D}^{3}(x,y) \ $ be the set of continuous diagrams
with width less than or equal to  $\ x, \ $ and height less thatn or 
equal to $\ y . \ $\\

The east-north path associated to a diagram in $\ \mathrm{D}^{1}(x,y) \ $
can be extended in a unique way, by adding a last east step, so that the path actually ends up at $ \ (x,y) .\ $ From this, and related observations,  we obtain that  
$$\ \mathrm{D}^{i}(x,y) = \coprod_{n=1}^{\infty}\mathrm{D}_n^{i}(x,y)
\ =  \ \coprod_{n=1}^{\infty}\{ D_p  \ | \ p \in \mathrm{P}_n^{i}(x,y) \}, 
\ \ \ \ \mbox{where}$$

\begin{itemize}
  \item A directed path $\ p \ $ is in  $ \  \mathrm{P}_n^{1}(x,y)\ $ if it has length $\ 2n+1, \ $ begins at $\ (0,0), \ $ ends at $ \ (x,y), \ $ moves in alternating fashion east or north, starts and ends with an east step. Thus $\ \mathrm{P}_n^{1}(x,y) \ \simeq \ \Delta^{n}_x \times \Delta^{n-1}_y.$
  \item A directed path $\ p \ $ is in  $ \ \mathrm{P}_n^{2}(x,y) \ $ if it has length $\ 2n+1, \ $ begins at $\ (0,0), \ $ ends at $ \ (x,y), \ $ moves in alternating fashion east or north, starts and ends with a north step.
      Thus $\ \mathrm{P}_n^{2}(x,y) \ \simeq \ \Delta^{n-1}_x \times \Delta^{n}_y.$
  \item A directed path $\ p \ $ is in  $ \ \mathrm{P}_n^{3}(x,y) \ $ if it has length $\ 2n+2, \ $ begins at $\ (0,0), \ $ ends at $ \ (x,y), \ $ moves in alternating fashion east or north, starts with a north step and ends with an east step. Thus $\ \mathrm{P}_n^{3}(x,y) \ \simeq \ \Delta^{n}_x \times \Delta^{n}_y.$
\end{itemize}

\

\begin{thm}\label{yc}
{\em   With the preceding notation we have for $\ x>0, \ y> 0 \ $ that:
  \begin{description}
    \item[a)] $ \displaystyle \mathrm{vol}(\mathrm{D}^{1}(x,y)) \ = \ 
    \int_{0}^{x}\mathrm{vol}(\mathrm{D}(s,y))ds \ = \ \sqrt{\frac{x}{y} } \ I_1(2\sqrt{xy}).$
    \item[b)] $\displaystyle \mathrm{vol}(\mathrm{D}^{2}(x,y)) \ = \ 
    \int_{0}^{y}\mathrm{vol}(\mathrm{D}(x,t))dt \ = \ \sqrt{\frac{y}{x} } \ I_1(2\sqrt{xy}).$
    \item[c)] $\displaystyle \mathrm{vol}(\mathrm{D}^{3}(x,y)) \ = \ 
    \int_{0}^{x}\int_{0}^{y}\mathrm{vol}(\mathrm{D}(s,t))dtds \ = \ \mathrm{vol}(\mathrm{D}(x,y)) -1.$
  \end{description}
}
\end{thm}

\begin{proof} Other items being similar, we just proof item a). In the last equation below we use Theorem \ref{t3}-a),d), we have 
$$\mathrm{vol}(\mathrm{D}^{1}(x,y)) \ = \  \mathrm{vol}(\mathrm{P}^{1}(x,y))
\ = \  \mathrm{vol}(\coprod_{n=1}^{\infty}\mathrm{P}_n^{1}(x,y)) \ = $$
$$\sum_{n=1}^{\infty} \mathrm{vol}(\mathrm{P}_n^{1}(x,y))  \ = \ 
\sum_{n=1}^{\infty} \mathrm{vol}(\Delta_x^{n}\times \Delta_y^{n-1}) \ = 
\sum_{n=1}^{\infty} \mathrm{vol}(\Delta_x^{n})\mathrm{vol}(\Delta_y^{n-1}) \ =$$
$$\sum_{n=0}^{\infty} \frac{x^{n+1}}{(n+1)!}\frac{y^{n}}{n!} \ = \ \int_{0}^{x}\mathrm{vol}(\mathrm{D}(s,y))ds \ = \ \sqrt{\frac{x}{y} } \ I_1(2\sqrt{xy}).$$
\end{proof}

\

Next we collet some additional properties of the function
$\ \mathrm{vol}(\mathrm{D}(x,y)). $\\

\begin{thm}
{\em For $\ x>0, \ y> 0 \ $ the following properties hold.
\begin{description}
  \item[a)] For $\ xy \rightarrow \infty, \  k \geq 0, \ $ there are constants $\ c_k \in \mathbb{Q} \ $ such that  we have the asymptotic behaviour 
   $$ \frac{\partial^k}{\partial x^k} \mathrm{vol}(\mathrm{D}(x,y)) \  \sim \ 
   \frac{c_k}{2\sqrt{\pi}}
   x^{-\frac{k}{2}-\frac{1}{4}}y^{\frac{k}{2}-\frac{1}{4}}e^{2\sqrt{xy}},
   $$
    $$\frac{\partial^k}{\partial y^k} \mathrm{vol}(\mathrm{D}(x,y)) \  \sim \ 
   \frac{c_k}{2\sqrt{\pi}}
   x^{\frac{k}{2}-\frac{1}{4}}y^{-\frac{k}{2}-\frac{1}{4}}e^{2\sqrt{xy}}.$$
  \item[b)] The following integral representations hold
  $$ \displaystyle \frac{\partial^k}{\partial x^k}\mathrm{vol}(\mathrm{D}(x,y))  \ = \  \frac{1}{\pi}\bigg(\frac{y}{x}\bigg)^{\frac{k}{2}}
  \int_{0}^{\pi}e^{2\sqrt{xy}\mathrm{cos}(\theta)}\mathrm{cos}(k\theta)d\theta,$$
  $$\frac{\partial^k}{\partial y^k}\mathrm{vol}(\mathrm{D}(x,y))  \ = \  \frac{1}{\pi}\bigg(\frac{x}{y}\bigg)^{\frac{k}{2}}
  \int_{0}^{\pi}e^{2\sqrt{xy}\mathrm{cos}(\theta)}\mathrm{cos}(k\theta)d\theta. $$
  \item[c)] For $ \ k \geq 1 \ $ the following formula holds
  $$\mathrm{vol}(\mathrm{D}(x,y))^k \ = \ 
   \sum_{n=0}^{\infty}\bigg( \sum_{n_1+\cdots + n_k=n} \binom{n}{n_1,...,n_k}^2 \ \bigg)\frac{x^ny^n}{n!n!}, \ \ \ \ \ \mbox{in particular}$$
  $$\mathrm{vol}(\mathrm{D}(x,y))^2 \ = \ 
   \sum_{n=0}^{\infty}\bigg( \sum_{k=0}^n \binom{n}{k}^2 \ \bigg)\frac{x^ny^n}{n!n!} \ = \ \sum_{n=0}^{\infty}\binom{2n}{n}\frac{x^ny^n}{n!n!}.$$
  \item[d)] The following formula holds
  $$\frac{1}{\mathrm{vol}(\mathrm{D}(x,y))}\ = \ 
   1\ + \ \sum_{n=1}^{\infty}\bigg(
   \sum_{\substack{n_1+\cdots + n_k=n \\ k\geq 1, \ n_i\geq 1}}
    (-1)^k\binom{n}{n_1,...,n_k}^2 \ \bigg)\frac{x^ny^n}{n!n!}.$$
\end{description}

}
\end{thm}

\begin{proof}
Items a) and b) follow from known properties  of the Bessel functions
\cite{n}. Items c) and d) follow from standard techniques for power series. 
\end{proof}

\

A continuous analogue for the binomial coefficients $\ \displaystyle { x \brace s } \ $ was built in  \cite{cd, dc} using the techniques summarized in the introduction, and has been further studied in \cite{lrvw, rod, vw}. Next result 
writes these coefficients in terms of $\ \mathrm{vol}(\mathrm{D}(x,y)). \ $
Conversely, $\ \mathrm{vol}(\mathrm{D}(x,y)) \ $ can be obtain from the continuous  binomial coefficients via a differential equation.\\

\begin{thm}
{\em Set $\ \rho(x,y)=\mathrm{vol}(\mathrm{D}(x,y)). \ $ For $\ 0 \leq s \leq x \ $ the continuous binomial coefficients
are given by
$${ x \brace s } \ = \ 2\rho(s,x-s) \ + \
 \frac{x}{x-s}\frac{\partial \rho}{\partial x }(s,x-s) .$$   

}
\end{thm}

\begin{proof}
Use Theorem \ref{t3} above,  and Definition 6, Identity 7  from \cite{cd}.
\end{proof}

\

\begin{cor}
{\em  For $\ x>0, \ y> 0 \ $ the function $ \ \rho(x,y)=\mathrm{vol}(\mathrm{D}(x,y)) \ $ satisfies the first order non-homogeneous linear differential equation 
$$(x+y)\frac{\partial \rho}{\partial x } \ + \ 
2y\rho \ - \ y { x+y \brace x } \ = \  0.$$   

}
\end{cor}

\

\section{$ q\ $ and $ \ ln(q)\ $ generating functions}

The cardinality of finite sets can be generalized for finite sets with weights in a commutative monoid $ \ w:x\rightarrow A\ $ as 
$\ \ |(x,w)|_A=\sum_{i\in x}w(i). \ \ $ An important case is
the $\ q$-cardinality for $\ \mathbb{N} $-graded finite sets $\ w:x\rightarrow \mathbb{N} \ $ which is given by 
$$\ |(x,w)|_q\ = \ \sum_{i\in x}q^{w(i)}\ = \ \sum_{n=0}^{\infty}|w^{-1}(n)|q^n \  \in \ \mathbb{N}[q].\ $$
Choosing an element $\ q\in A, \ $ then $\ |(x,w)|_q \ $ may be regarded as taking values in $ \ A. \ $ The usual choice is to let $\ 0< q<1 \ $ be a real number, but other choices are also interesting. 
For a subset $\ x \subseteq \mathbb{N} \ $ the grading $ \ w \ $ is just the inclusion map, for example for $\ n\geq 1 \ $ we have that
$$\ \displaystyle [n]_q\ := \ |\{0,...,n-1\}|_q\ = \ 1+\cdots +q^{n-1} \ =\ \frac{1-q^n}{1-q}. \ $$ 
As a second example, consider the group of permutations $\ S_n \ $ as a 
finite set with $\ \mathbb{N} $-grading $\ \mathrm{inv}:S_n \rightarrow \mathbb{N}\ $ sending a permutation $\ \alpha \ $ to the number 
$\ \mathrm{inv}(\alpha) \ $ of inversion of $\ \alpha. \ $ The $\ q$-cardinality of $\ (S_n, \mathrm{inv})\ $ is the $\ q$-analogue
$\ [n]_q!\ $ of the factorial number $\ n!, \ $ namely
$$|(S_n, \mathrm{inv})|_q \ = \  \sum_{\alpha \in S_n } q^{ \mathrm{inv}(\alpha)}  \ = \ \prod_{k=1}^{n}[k]_q \ = \ [n]_q!.$$
Note that $\ q$-cardinality may often be applied to infinite sets, for example, it can always be computed for subsets of $\ \mathbb{N}. \ $ As an example we have that
$$\displaystyle \frac{|l\mathbb{N}|_q}{|\mathbb{N}|_q}\ = \ \frac{1+\cdots+q^{ln}+\cdots }{1+\cdots+q^{n}+\cdots} \ = \ 
\frac{\frac{1}{1-q^l} }{\frac{1}{1-q}}\ = \ \frac{1\ }{[l]_q}.$$
The substitution $\ q=e^{-z} \ $ transforms polynomial identities in $\ q \ $
into identities involving power series in $\ z, \ $ for example for
$\ \displaystyle 1+\cdots +q^{n-1}=\frac{1-q^n}{1-q\ }\ $ we get
$$1\ +\ \sum_{k=0}^{\infty}\bigg(1^n +\cdots+(n-1)^k\bigg)\frac{(-z)^k}{k!}\ = \ \frac{1-e^{-zn}}{1-e^{-z}},$$ i.e. we obtain an identity for the generating series for the power sum numbers in the variable $\ -z. \ $ For infinite sums the substitution $\ q=e^{-z} \ $ may be problematic, for example setting $\ q=e^{-z} \ $  in  
$ \ 1+\cdots+q^{n}+\cdots\ $ does not yield a power series in  $\ z. $ \\

In summary, the $\ z$-cardinality of a finite 
$\ \mathbb{N}$-graded set is given by
$$\ |(x,w)|_z \ = \ 
\sum_{n=0}^{\infty}|w^{-1}(n)|e^{-nz} \ = \  
\sum_{i\in x}e^{-w(i)z}\ = \ \sum_{n=0}^{\infty}\bigg(\sum_{i\in x}w(i)^n \bigg)\frac{(-z)^n}{n!}  \  \in \ \mathbb{Z}[[z]].$$ For 
$\ z\in \mathbb{R}, \ $ the usual convention is to take $\ z>0,\ $ the sum above produces a real number, however, we emphasize that $\ z$-cardinality makes sense as a power series, i.e. $\ |(x,w)|_z\ $ is an entire analytic function in the variable $\ -z, \ $  furthermore $\ |(x,w)|_0= |x|,\ $ thus  $\ |(x,w)|_z\ $ is a one-parameter analytic deformation of  $\ |x| \ $ with higher order coefficients   $\ |(x,w^n)|.\ $  \\

\begin{exmp}
{\em  $\ \mathrm{Y}(m,n) \ $ is a  $\ \mathbb{N}$-graded set with the map
$\ \mathrm{area}: \mathrm{Y}(m,n) \rightarrow \mathbb{N}.\ $  The $\ q$-cardinality
of $\ (\mathrm{Y}(m,n),\mathrm{area}) \ $ is given by
$$|\mathrm{Y}(m,n)|_q  \ = \ \sum_{\lambda \in \mathrm{Y}(m,n)}q^{\mathrm{area}(\lambda)} \ = \ \sum_{a=m+n-1}^{mn}|\mathrm{Y}(m,n,a)| q^{a}.$$ The polynomials 
$\ |\mathrm{Y}(m,n)|_q \ \in \ \mathbb{N}[q]\ $ are determined by the recursion
$$|\mathrm{Y}(m,1)|_q=q^m \ \ \ \ \ \mbox{and} \ \ \ \ \ 
|\mathrm{Y}(m,n+1)|_q=q^m\sum_{l=1}^{m} |\mathrm{Y}(l,n)|_q.$$
On the other hand the $\ z$-cardinality of $\ (\mathrm{Y}(m,n),\mathrm{area})\ $ is given by
$$|\mathrm{Y}(m,n)|_z \ = \
\sum_{\lambda \in \mathrm{Y}(m,n)}e^{-\mathrm{area}(\lambda)z} \ = \
\sum_{k=0}^{\infty}\bigg(\sum_{\lambda \in \mathrm{Y}(m,n)} \mathrm{area}(\lambda)^k\bigg)\frac{(-z)^k}{k!} , $$
thus $\ |\mathrm{Y}(m,n)|_{z}\ $ is the generating series,
in the variable $ \ -z,\ $ of the sum of powers of the areas of Young diagrams.
}
\end{exmp}

\
  
We call $\ z$-volume our continuous analogue for the notion of $\ z$-cardinality, it is a one parameter deformation of the notion of volume of a manifold built as follows. Let $\ (M,dx,f) \ $ consist of an oriented finite dimensional compact manifold with corners $ \ M,\ $ a volume form $\ dx \ $ on $\ M \ $ compatible with the orientation of $\ M,\ $  and a smooth map $\ f: M \rightarrow \mathbb{R}_{\geq 0}. \ $ The $\ z$-volume 
of $\ (M,dx,f) \ $ is given by
$$ \ \mathrm{vol}_z(M) \ = \ \int_Me^{-zf(x)}dx \ = 
\ \sum_{n=0}^{\infty}\bigg( \int_Mf(x)^ndx \bigg)\frac{(-z)^n}{n!} \ \in
\ \mathbb{R}[[z]].$$
For  $\ z\in \mathbb{R} \ $ the integral above yields a real number, however, we emphasize that these numbers come from a power series, i.e. $ \ \mathrm{vol}_z(M) \ $ is the generating series
for the integrals of powers of $\ f(x),\ $  it is an entire analytic function on the variable $\ -z, \ $  furthermore $ \ \mathrm{vol}_0(M)=\mathrm{vol}(M), \ $
and $ \ \mathrm{vol}_z(M) \leq  \mathrm{vol}(M) \ $ for $\ z\in \mathbb{R}_{\geq 0}. \ $  Note that $\ f \ $ is bounded since it is continuous and $\ M \ $ is compact, therefore, the power series
$\ \displaystyle e^{-zf(x)}= \sum_{n=0}^{\infty} \frac{(-zf(x))^n}{n!} \ $ is absolutely and uniformly convergent. Likewise, the power series for $ \ \mathrm{vol}_z(M) \ $  is absolutely convergent. \\

Recall that $\ \mathrm{area}: \mathrm{D}(x,y) \rightarrow \mathbb{\mathbb{R}}_{\geq 0} \ $
assigns to each diagram the area of its underlying region.\\

\begin{lem}\label{l8}
{\em For $ \  (n,l) \in \mathbb{N}_{\geq 1}\times \mathbb{N} \ $ consider  the  rational number
$$d_{n,l} \ = \ \sum_{l_1+\cdots+l_n=l}
 \frac{1}{\prod_{i=1}^{n}\big(|l|_i+i\big)} \ \ \ \ \mbox{where} \ \ \ l_i\in \mathbb{N}.$$
The numbers $\ d_{n,l} \ $ are determined by 
$$d_{n,0}\ = \ \frac{1}{n!}, \   \ \ \ \ \  d_{1,l}\ = \ \frac{1}{l+1}, \ \ \ \ \ \  \mbox{and}\ \  \ \ \ \
d_{n+1,l}\ = \ \frac{1}{n+l+1}\sum_{i=0}^{l}d_{n,i} .$$ 
We have that
$\ \ \displaystyle d_{n,l} \ \leq \  \frac{1}{n!}\binom{l+n-1}{n-1} \ = \ 
\frac{(l+n-1)!}{n!(n-1)!l!}.$
}
\end{lem}

\begin{proof}
For the latter statement we have that
$$\sum_{l_1+\cdots+l_n=l}
\frac{1}{\prod_{i=1}^{n}\big(|l|_i+i\big)}\ \leq \ \sum_{l_1+\cdots+l_n=l}\frac{1}{\prod_{i=1}^{n}i}\ = \ \frac{1}{n!}\sum_{l_1+\cdots+l_n=l}1 \ = \
\frac{1}{n!}\binom{l+n-1}{n-1}.$$ 
\end{proof}

\

\begin{thm}\label{22}
{\em For $\ x>0,\ y>0 \ $ the $\ z$-volume 
  of  $\ (\mathrm{D}_n(x,y),\ dxdy,\ \mathrm{area}) \ $ is given by $$\ \mathrm{vol}_{z}(\mathrm{D}_1(x,y)) \ = \ e^{-xyz}, \ $$ $$ \mathrm{vol}_z(\mathrm{D}_n(x,y)) \ = \
  x^{n-1}y^{n-1}\sum_{l=0}^{\infty} \frac{(l+n)d_{n,l}}{(l+n-1)!} (-xyz)^l 
  \ \ \ \ \mbox{for} \ \ \ \ n \in \mathbb{N}_{\geq 2}.$$
}  
\end{thm}

\begin{proof} 
The space $\ \mathrm{D}_1(x,y) \ $ has a unique point, namely, the diagram of area $\ xy \ $ associated to the path $ \ (0,0)\rightarrow (x,0) \rightarrow (x,y). \ $ Integration is  evaluation, thus
$\ \mathrm{vol}_{z}(\mathrm{D}_1(x,y)) =e^{-xyz}. \ $ For $\ n \in \mathbb{N}_{\geq 2} \ $ we have that
$$ \ \mathrm{vol}_z(\mathrm{D}_n(x,y)) \ = \ \int_{\mathrm{D}_n(x,y)}e^{-z  \ \mathrm{area}}\ dxdy \ = \
\int_{\Delta^{n-1}_x \times \Delta^{n-1}_y}e^{-z\big(\sum_{i=1}^{n}x_i(y_i-y_{i-1}) \big)}dxdy \ = $$
$$ \sum_{l=0}^{\infty} \bigg( \sum_{l_1+\cdots+l_n=l}\binom{l}{l_1,...,l_n}
\bigg( x^{l_n}\int_{\Delta^{n-1}_x }\prod_{i=1}^{n-1}x_i^{l_i}dx \bigg)   
\bigg( \int_{\Delta^{n-1}_y }\prod_{i=1}^{n}(y_i-y_{i-1})^{l_i}dy \big)\bigg)\bigg) \frac{(-z)^l}{l!} \ = $$
$$ \sum_{l=0}^{\infty} \bigg( \sum_{l_1+\cdots+l_n=l}l!
x^{l_n}\bigg(\int_{\Delta^{n-1}_x }\prod_{i=1}^{n-1}x_i^{l_i}dx \bigg)   
\bigg( \int_{\Delta^{n-1}_y }\prod_{i=1}^{n}(y_i-y_{i-1})^{(l_i)}dy \big)\bigg) \bigg)\frac{(-z)^l}{l!} \ = $$
$$ \sum_{l=0}^{\infty} \bigg( \sum_{l_1+\cdots+l_n=l}
 \frac{x^{l+n-1}}{\prod_{i=1}^{n-1}\big(|l|_i+i\big)}\frac{y^{l+n-1}}{(l+n-1)!}\bigg) (-z)^l \ = $$
$$ x^{n-1}y^{n-1}\sum_{l=0}^{\infty}\bigg( \sum_{l_1+\cdots+l_n=l}
 \frac{1}{\prod_{i=1}^{n-1}\big(|l|_i+i\big)}\bigg) \frac{1}{(l+n-1)!}  (-xyz)^l\ = \ $$
$$  x^{n-1}y^{n-1}\sum_{l=0}^{\infty} \frac{(l+n)d_{n,l}}{(l+n-1)!} (-xyz)^l.$$ 
\end{proof}

\

\begin{thm}
{\em For $\ x>0,\ y>0 \ $ the $\ t$-volume 
 of  $\ (\mathrm{D}(x,y),\ dxdy,\ \mathrm{area}) \ $ is given by
$$ \mathrm{vol}_{z}(\mathrm{D}(x,y))  \ = \ e^{-xyz} \ + \ \sum_{k=0}^{\infty}(k+1)\bigg(\sum_{l=0}^{k-1} d_{k-l+1,l} (-z)^l\bigg)\frac{(xy)^{k}}{k!}.$$
}  
\end{thm}

\begin{proof} By Theorem \ref{22} we have that
$$ \ \mathrm{vol}_z(\mathrm{D}(x,y)) \ = \  
\mathrm{vol}_z(\coprod_{n\geq 1}\mathrm{D}_n(x,y)) \ = \ 
 \sum_{n\geq 1}\mathrm{vol}_z(\mathrm{D}_n(x,y))
 \ = $$ 
 $$e^{-xyz} \ + \ \sum_{n=2}^{\infty}\mathrm{vol}_z(\mathrm{D}_n(x,y))\ = \
1 \ + \ \sum_{n=2}^{\infty}x^{n-1}y^{n-1}\sum_{l=0}^{\infty} \frac{(l+n)d_{n,l}}{(l+n-1)!} (-xyz)^l \ = $$
$$e^{-xyz} \ + \ \sum_{n=1}^{\infty}\sum_{l=0}^{\infty} \frac{(l+n+1)d_{n+1,l}}{(l+n)!}  x^{l+n}y^{l+n}(-z)^l.$$

\

\noindent  By Lemma \ref{l8} the sum above is absolutely convergent since

$$\frac{(l+n+1)d_{n+1,l}}{(l+n)!}x^{l+n}y^{l+n}z^l
\ \leq \ \frac{(l+n+1)}{n!\ n! \ l!} 
 (xy)^{n}(xyz)^l \ = \ (l+n+1)
 \frac{(xy)^{n}}{n! n!}\frac{(xyz)^l}{l!}.$$

\

\noindent Setting $\ k=n+l \ $ in the series above we get that 

$$ \mathrm{vol}_z(\mathrm{D}(x,y)) \ = \  e^{-xyz} \ + \ \sum_{k=1}^{\infty}(k+1)\bigg(\sum_{l=0}^{k-1} d_{k-l+1,l}(-z)^l \bigg)\frac{(xy)^{k}}{k!}. $$

\

\noindent Finally, we note  that $ \ \mathrm{vol}_z(\mathrm{D}(x,y)) \ $ may alternatively be written as 
$$\mathrm{vol}_z(\mathrm{D}(x,y)) \ = \  e^{-xyz} \ + \ 
\sum_{l=0}^{\infty} \bigg( \sum_{n=1}^{\infty}
\frac{(l+n+1)d_{n+1,l}(xy)^{n+l}}{(l+n)_{(n)}}\bigg)\frac{(-z)^l}{ l!} .$$

\end{proof}

\begin{cor}
{\em For $\ x>0,\ y>0, \  z\geq 0 \ $ the following inequalities hold
$$\ \ \mathrm{vol}_z(\mathrm{D}(x,y)) \ \leq \  
I_0(2\sqrt{xy}) \ \ \ \ \  \mbox{and} \ \ \ \ \ 
  \mathrm{vol}_{z}(\mathrm{D}(\frac{x}{2},\frac{x}{2}))  \ \leq \ I_0(x). $$
}  
\end{cor}

\begin{proof} 
We have that
$$ \ \mathrm{vol}_z(\mathrm{D}(x,y)) \ = \  
\mathrm{vol}_z(\coprod_{n\geq 1}\mathrm{D}_n(x,y)) \ = \ 
 \sum_{n\geq 1}\mathrm{vol}_z(\mathrm{D}_n(x,y))
 \ \leq \ $$
 $$ \sum_{n\geq 1}\mathrm{vol}(\mathrm{D}_n(x,y)) \ = \
  \mathrm{vol}(\mathrm{D}(x,y)) \ = \  I_0(2\sqrt{xy}).$$

\end{proof}

\section{Open problems}

A diagram in $ \  \mathrm{YD}(m,n,a) \ $ represents a partition of $ \ a \ $ with
width  $\ m \ $  and  height (and thus number of blocks) $\ n. \ $ On the continuous side height and number of blocks need not be equal, so we are led to consider the spaces $\ \mathrm{D}_n(x,y,a) \ $ with fix width $\ x, \ $ height $\ y,\ $ number of blocks $\ n, \ $ and area $\ a.$\\

\begin{prop}
{\em For $\ x>0, \ y>0,\ n\geq 2 \ $ the smooth map 
$\ \mathrm{area}: \mathrm{D}_n(x,y) \longrightarrow \mathbb{R}_{\geq 0}\ $
has a unique critical point 
$\ (x,..., x \ ; \ 0, ..., 0) \ \in \ \Delta^{n-1}_x \times \Delta^{n-1}_y 
\ = \ \mathrm{D}_n(x,y).$
  
}
\end{prop} 

\begin{proof}
Recall that $\  \mathrm{area} \ = \ \sum_{i=1}^nx_i(y_i-y_{i-1}),\  $
thus we have that
$$d(\mathrm{area}) \ = \ \sum_{i=1}^{n-1}\frac{\partial \mathrm{area} }{\partial x_i}dx_i \ + \ \sum_{i=1}^{n-1}\frac{\partial \mathrm{area} }{\partial y_i}dy_i \ =  \
\sum_{i=1}^{n-1}(y_i-y_{i-1})dx_i \ + \ \sum_{i=1}^{n-1}(x_i-x_{i+1})dy_i.$$
Thus $\ d(\mathrm{area})=0 \ \ $ if and only if $ \ \ x_1=\cdots = x_{n-1}=x \ \ $ and $\ \ y_1=\cdots = y_{n-1}=0. \ $
\end{proof} 
 
\
 
Note that  $\ \ \mathrm{area}^{-1}(0)\ = \ \bigcup_{l=0}^{n-1} \mathrm{area}_l^{-1}(0) \ \ \ \  \mbox{where} \ \ $  $$(x_1,...x_{n-1}\ ; \ y_1,...y_{n-1}) \in \mathrm{area}_l^{-1}(0) \ \ \ \ \mbox{if and only if} \ \ \ $$
$$y_{n-1}=y, \ \ \ \  \ x_i=0 \ \ \ \mbox{if} \ \ \ i\leq l,\ \ \ \ \ \mbox{and} \ \ \ \ \ y_i=y_{i-1} \ \ \ \mbox{if} \ \ \ i>l. $$
So $\ \mathrm{area}^{-1}(0) \ $ is included in the codimension $\ n \ $ boundary of $\ \mathrm{D}_n(x,y) . \ $  Furthermore the critical point 
$\ (x,..., x \ ; \ 0, ..., 0) \ $ lies in  $ \ \mathrm{area}^{-1}(0). \ $

\

Similarly
$\ \ \ \mathrm{area}^{-1}(xy)\ = \   \bigcup_{l=1}^{n}\mathrm{area}_l^{-1}(xy) \ \ \ \ \  \   \mbox{where}  $  $$(x_1,...,x_{n-1}\ ; \ y_1,...,y_{n-1}) \in \mathrm{area}_l^{-1}(xy) \ \ \ \ \mbox{if and only if}$$
$$x_{i}=x \ \ \ \mbox{if} \ \ \ i\geq l \ \ \ \ \mbox{and} \ \ \ \ 
y_{i}=y_{i-1} \ \ \ \mbox{if} \ \ \ i< l. $$
So $\ \mathrm{area}^{-1}(xy) \ $ is included in the codimension $\ n-1 \ $ boundary of $\ \mathrm{D}_n(x,y) . \ $  

\

For $\ 0< a < xy,\ $ since there are no critical points in 
$\ \mathrm{D}_n(x,y,a) = \mathrm{area}^{-1}(a),\ $ this set is the closure of a smooth codimension one submanifold of  $ \ \mathrm{D}_n(x,y)\ $ in the open part, and as such it acquires a Riemannian metric and an orientation.
Thus $ \  \mathrm{vol}(\mathrm{D}_n(x,y,a))\ $ measures continuous diagrams of fix width, height, area, and number of blocks. Within  our settings an analogue for $ \  |\mathrm{YD}(m,n,a)| \ $ should be given by the sum 
$\ \displaystyle  \mathrm{vol}(\mathrm{D}(x,y,a))  =  \sum_{n=1}^{\infty} \mathrm{vol}(\mathrm{D}_n(x,y,a)), \ $ however, weather this sum is convergent or not is left open.\\

For $\ 0 < w \leq xy\ $ set  $\ \widehat{\mathrm{D}}_n(x,y,w) = \mathrm{area}^{-1}[0,w].\ $ Note that $\ \widehat{\mathrm{D}}_n(x,y,w) \ $ is 
the closure of the  open region $\ \mathrm{area}^{-1}(0,w) \ $ in $\ \mathrm{D}_n(x,y), \ $ therefore
$$\mathrm{vol}(\widehat{\mathrm{D}}_n(x,y,w))  \ \leq \  \mathrm{vol}(\mathrm{D}_n(x,y)) \ = \ \frac{x^{n-1}}{(n-1)!}\frac{y^{n-1}}{(n-1)!}, $$ and the following function is well defined
$$\mathrm{vol}(\widehat{\mathrm{D}}(x,y,w)) \ = \ \mathrm{vol}(\coprod_{n=0}^{\infty}\widehat{\mathrm{D}}_n(x,y,w)) \ = \ 
\sum_{n=1}^{\infty}\mathrm{vol}(\widehat{\mathrm{D}}_n(x,y,w)).$$
Note that  $\ \mathrm{vol}(\widehat{\mathrm{D}}(x,y,w)) \ $ may be regarded as continuous analogue for $\ \sum_{l=n+m-1}^{a}\big|\mathrm{Y}(n,m,l)\big|, \  $ the number of partitions of integers less than or equal to $\ a, \ $ with width $\ n,\ $ and largest block $\ m. \ $ 
A detailed study of  $\ \mathrm{vol}(\mathrm{D}_n(x,y,w)), \ \mathrm{vol}(\mathrm{D}(x,y,a)), \   \mathrm{vol}(\widehat{\mathrm{D}}_n(x,y,w)) \ $ and $\ \mathrm{vol}(\widehat{\mathrm{D}}(x,y,w)) $ is left for the future.

\

\

\

\noindent ragadiaz@gmail.com\\
\noindent Departamento de Matem\'aticas \\
\noindent Universidad Nacional de Colombia - Sede Medell\' in\\
\noindent Colombia\\

\end{document}